\documentclass{article}
\def\withcolors{1}
\def\withnotes{1}
\usepackage{ccanonne}

\title{Gaussian Mean Testing under Truncation}
\author{
Cl\'{e}ment L. Canonne\\
University of Sydney \\
\texttt{clement.canonne@sydney.edu.au} 
\and
Themis Gouleakis \\
Nanyang Technological University \\
\texttt{themis.gouleakis@ntu.edu.sg}
\and
Yuhao Wang\\
National University of Singapore\\
\texttt{yuhaowang@u.nus.edu}
\and
Joy Qiping Yang\\
University of Sydney \\
\texttt{qyan6238@uni.sydney.edu.au}
}
\date{}

\begin{document}
\maketitle
\begin{abstract}
  We consider the task of Gaussian mean testing, that is, of testing whether a high-dimensional vector perturbed by white noise has large magnitude, or is the zero vector. This question, originating from the signal processing community, has recently seen a surge of interest from the machine learning and theoretical computer science community, and is by now fairly well understood. What is much less understood, and the focus of our work, is how to perform this task under truncation: that is, when the observations (i.i.d.\ samples from the underlying high-dimensional Gaussian) are only observed when they fall in an given subset of the domain $\R^d$. This truncation model, previously studied in the context of learning (instead of testing) the mean vector, has a range of applications, in particular in Economics and Social Sciences. As our work shows, sample truncations affect the complexity of the testing task in a rather subtle and surprising way.
\end{abstract}

\section{Introduction}

The Gaussian mean testing problem, which originated in the context of signal processing under the name of signal detection, asks the following: given independent observations from a high-dimensional vector subject to random white noise, how to detect whether the underlying signal has large magnitude, or is non-significant? This can be seen as a hypothesis testing version of the so-called \emph{Gaussian location model} (GLM) question from information theory and signal processing, where the objective is to \emph{detect} a signal instead of \emph{learning} it.

Mean testing has recently seen a surge of interest from the machine learning and theoretical computer science (and, specifically, \emph{distribution testing}) communities, focusing on efficient algorithms with finite-sample guarantees, i.e., requiring as few observations (samples) as possible. This culminated in simple, sample-optimal algorithms for this task under an array of settings, including relaxing the assumption on the random noise~\cite{canonne2021random,diakonikolas2022gaussian}, considering it in the distributed, communication-limited setting~\cite{AcharyaCT20,SzaboVZ23}, or requiring robustness to adversarial corruptions of the observations~\cite{canonne2023full}.

In this work, we consider a different variant, and focus on the \emph{truncated samples} setting. Truncation happens when some observations fail to be observed or recorded, e.g., due to limitations in the sensing equipment or, in the case of social studies or surveys, when a subset of respondents systematically withhold their response. A typical example is when asking insurance customers for some sensitive medical information, as people with at-risk factors may decide to opt out of the survey entirely for fear of having their insurance premiums go up. Truncated samples (and the related notion of censored data) have a rich history in Statistics, and a host of applications in medical science, social studies, and Economics, to name a few (see, e.g.,~\cite{cohen1991truncated}); and, following~\cite{daskalakis2018efficient}, has recently been the focus of a line of work on efficient truncated statistics, whereby one seeks to develop efficient algorithms to efficiently estimate the parameters of a population given truncated samples: we elaborate on this in~\cref{sec:related}.

Despite the existence of these two lines of work~--~one on Gaussian mean testing, and the other on learning parameters from truncated samples, to the best of our knowledge there has not been any study of the very natural related question of \emph{Gaussian mean testing from truncated samples.} In this work, we address this question, and show that the complexity of the testing task changes drastically (and quite surprisingly) depending on the truncation set itself, and whether we have some \emph{a priori} information about it. In order to present our results and discuss their implications, we start by formally defining the problem:

\paragraph{Problem formulation.}
Let $\tmmathbf{\mu} \in \R^d, \tmmathbf{\Sigma}\in \R^{d\times d}$ be an unknown vector and covariance matrix, respectively, and $S \subseteq \R^d$, the \emph{truncation set}, be a subset of measure at least $1-\eps$ under the spherical normal distribution $\cN(\tmmathbf{\mu}, \tmmathbf{\Sigma})$, where $0 \leq \eps < 1$. We define the $S$-truncated Gaussian distribution, denoted $\cN(\tmmathbf{\mu}, \tmmathbf{\Sigma}, S)$, as the normal distribution $\cN(\tmmathbf{\mu}, \tmmathbf{\Sigma})$ conditioned on taking values on the subset $S$. We suppose that samples, $\mb{X} = \{\tmmathbf{x^{(1)}}, \dots, \tmmathbf{x^{(n)}}\}$, from an unknown $d$-variate normal $\cN(\tmmathbf{\mu}, \tmmathbf{\Sigma})$ are only revealed if they fall into some subset $S \in \R^d$; otherwise the
samples are hidden and their count in proportion to the revealed samples is also hidden.
We will make no assumptions about $S$, except that its measure $\eps$ with respect to the unknown distribution is non-trivial, say $\eps = 1\%$: that is, one should think of $\eps$ as a small (positive) constant. We will focus on the case of \emph{spherical} covariance matrices (before truncation), that is, where $\tmmathbf{\Sigma}=\tmmathbf{I_d}$: this corresponds to the signal detection problem alluded to before, where a signal is observed through random white noise.

Given $n$ i.i.d.\ samples $\tmmathbf{x}^{(1)}, \tmmathbf{x}^{(2)}, \dots$ from a truncated Gaussian distribution $P$ on $\R^d$ (with unknown vector $\tmmathbf{\mu}$ and truncation set $S$) and $\alpha \in (0,1]$ an accuracy, the task is to distinguish between the following cases: 
\begin{itemize}
   \item \textbf{(Completeness)} if $P = \cN(0,\tmmathbf{I_d}, S)$, the algorithm must output ``ACCEPT'' with probability at least $2/3$;
  \item \textbf{(Soundness)} if $P = \cN(\tmmathbf{\mu},\tmmathbf{I_d}, S)$ for some $\tmmathbf{\mu}$ with $\|\tmmathbf{\mu}\|_2 \geq \alpha$, the algorithm must output ``REJECT'' with probability at least $2/3$. 
\end{itemize}
The objective is to minimize the \emph{sample complexity} of the algorithm, i.e., the number of samples $n$ required to achieve the task, over all possible vectors $\tmmathbf{\mu}$ and truncation sets $S$.
Note that the complexity of the task might vary, depending on the parameter regime and the information available about $S$: namely, (1)~the relation between truncated mass $\eps$ and desired accuracy $\alpha$, and (2)~whether the set $S$ is unknown to the algorithm or known (either provided explicitly, or as a membership oracle.\footnote{A membership oracle for a set $S$ is a procedure which, on any input $x$, indicates whether $x\in S$.}

\subsection{Our contributions}
We establish upper and lower bounds on the sample complexity of the problem, and show it undergoes a stark transition as $\alpha$ and $\eps$ vary, when the truncation set is unknown to the algorithm. Specifically, we show the following, where, for ease of exposition, we focus on the dependence on the dimension $d$ and treat $\eps,\alpha$ as constants:
\begin{itemize}
   \item When $\eps \sqrt{\log 1/ \eps} \lesssim \alpha$, i.e., the accuracy parameter is significantly larger than the truncated probability mass, then the simple testing algorithm designed for the \emph{non-truncated} version of the problem works, achieving the optimal sample complexity $\Theta(\sqrt d)$ (\cref{mean_testing_small_truncation}).
   \item When $\eps \lesssim \alpha  \lesssim \eps \sqrt{\log 1/\eps}$, there is a sudden phase transition: we provide an information-theoretic lower bound showing that \emph{any} algorithm requires $\Omega(d)$ samples (\cref{lemma:lower_bound_truncation_near_accuracy}). Combined with an $O(d)$ upper bound obtained by \emph{learning} the unknown mean vector $\tmmathbf{\mu}$, our results show that in this regime \emph{testing suddenly becomes as hard as learning.}
    \item When $\alpha  \lesssim \eps$, it follows from \cite[Lemma 12]{daskalakis2018efficient} that the testing task becomes information-theoretically impossible, regardless of sample complexity.
\end{itemize}

\begin{table}[h]
    \centering
    \setlength{\tabcolsep}{16pt}
    \begin{tabular}{|c|c|c|c|}\hline
      & $\eps  \lesssim \frac{\alpha}{\sqrt{\log\frac{1}{\alpha}}}$ & $\frac{\alpha}{\sqrt{\log\frac{1}{\alpha}}}  \lesssim \eps  \lesssim \alpha$  & $\alpha  \lesssim \eps$\\\hline
     Unknown & \cellcolor{green!20}$\Theta{(\sqrt{d})}$ & \cellcolor{red!20} $\Theta{(d)}$ & \cellcolor{black!20}$\infty$\\\hline
     Known & \cellcolor{green!20}$\Theta{(\sqrt{d})}$ & \cellcolor{green!20}$\Theta{(\sqrt{d})}$  & \cellcolor{green!20}$\Theta{(\sqrt{d})}$\\\hline
    \end{tabular}
    \caption{Mean testing sample complexity for small enough constant $\eps$ and $\alpha$.}
    \label{tab:main_result_table}
\end{table}

In contrast, we show that when the truncation set is known, a different (yet still relatively simple) algorithm, based on the gradient of the maximum likelihood estimator, achieves the optimal sample complexity $O(\sqrt d)$, \emph{across all parameter ranges} (\cref{theorem:known_truncation_mean_testing}).

\subsection{Related Works}
\label{sec:related}
We here discuss the literature and previous related work.
\paragraph{Learning from Truncated or Censored Samples}
Distribution learning under censored, truncated mechanisms has had a long history. 
Censoring happens when the events can be detected, but the measurements (the values) are completely unknown, while truncation occurs when an object falling outside some subset are not observed, and their count in proportion to the observed samples is also not known, see \cite{deemer1955estimation, cohen1957solution, dixon1960simplified,  haas1990estimation,cohen1991truncated,barr1999mean, cha2013rethinking, charikar2017learning} for an overview of the related works in estimating the censored or truncated normal or other type of distributions. 
\cite{pearson1902systematic, pearson1908generalised,lee1914table} used the method of moments, while \cite{fisher1931properties} used the maximum likelihood approach for the distribution learning from truncated samples. Since then, 
\cite{daskalakis2018efficient, daskalakis2019computationally, daskalakis2020truncated} developed computationally and statistically efficient algorithms under the assumption that the truncation set is known. 
Furthermore, \cite{wu2019learning}
considered the problem of estimating the parameters of a $d$-dimensional rectified
Gaussian distribution from i.i.d. samples. This can be seen as a special case of the self-censoring truncation, where the truncation happens due to the ReLU generative model.

\paragraph{Testing if samples are truncated}
Orthogonally, \cite{DeNS23}, studied a different problem: Whether or not a set of i.i.d. samples from high-dimensional standard Gaussian has been previously truncated. They provide a positive answer for the setting in which the truncation is promised to be convex. In the follow-up work, \cite{DBLP:conf/stoc/De0NS24} studied the problem in a different setting: The distributions to test are from the class of Hypercontractive high-dimensional product distributions (includes standard Gaussian) and the truncation set can be characterized by a polynomial threshold function of degree $d$.

\paragraph{Robust mean estimation}Robust statistics~\cite{huber2011robust} considers statistical inference problems under the setting where samples observed could be contaminated in various ways.
For robust estimation, the usual goal is to obtain accurate estimation of parameters for parametric families such as Gaussian distributions under $\eps$-contamination, where $\eps$ is the maximum fraction of samples ($\eps \cdot n$ out of $n$) allowed to be contaminated. 
This problem has been extensively studied in recent years (see the book of \cite{diakonikolas2023algorithmic}, and references therein).
There are algorithms and lower bounds with different characteristics under different contamination models (time complexity and accuracy trade-off) \cite{DBLP:conf/colt/Hopkins019,DBLP:conf/colt/BlancLMT22,diakonikolas2019robust}.
\cite{hopkins2020robust} studied (nonparametric) robust mean estimation: distributions with finite covariance (see the survey by \cite{LugosiM19}, and references therein for more nonparametric works).
Notably, using algorithms developed through robust mean estimation (also called learning) for Gaussian under some strong contamination model, we can reduce our testing under truncation problem via the standard learning-to-test argument, which will give us a sample complexity upper bound of $\cO(d / \alpha^2)$.

\paragraph{Robust mean testing}
Gaussian mean testing has been studied and well known to have a sample complexity of $\Theta(\sqrt{d}/\alpha^2)$ \cite{diakonikolas2017statistical,diakonikolas2022gaussian}.
Recently, \cite{canonne2023full} studied the Gaussian mean testing problem under two contamination models: \emph{oblivious contamination model} and \emph{strong contamination model} -- both yield improved sample complexity than their learning counterparts. In the oblivious contamination model, an adversary could remove $\eps$ fraction of original samples from $P$ \emph{without} observing them and replacing them with samples from a different distribution. In this model, \cite{canonne2023full} prove a near-optimal sample complexity bound of $\tilde{\Theta} \left( \max \left( \frac{\sqrt{d}}{\alpha^2}, \frac{d
\varepsilon^3}{\alpha^4}, \min \left( \frac{d^{2 / 3} \varepsilon^{2 /
3}}{\alpha^{8 / 3}}, \frac{d \varepsilon}{\alpha^2} \right) \right) \right)$.

In the strong contamination model, where the adversary could first observe the values of original samples from $P$, then pick $\eps$ fraction of them and replace with arbitrary values, \cite{canonne2023full} give the near-optimal sample complexity bound of $\tilde{\Theta} \left( \max \left( \frac{\sqrt{d}}{\alpha^2}, \frac{d
\varepsilon^2}{\alpha^2} \right) \right)$\footnote{We use $\tilde{(\cdot)}$ to hide the polylogarithmic factors.}.

Indeed, truncation can be viewed as a special form of contamination model, and a strictly weaker form of contamination than the strong contamination model considered in 
\cite{canonne2023full}. Yet, it is somewhat orthogonal (neither stronger nor weaker) to the oblivious contamination model. We remark that our paper covers the full parameter regime in terms of the relation between $\eps$ and $\alpha$, while \cite[Theorem 7.1]{canonne2023full} has a limitation in the $\alpha \geq \eps \cdot \operatorname{polylog}(d, \frac{1}{\eps}, \frac{1}{\alpha})$. Under the regime, $\alpha \ll \eps \cdot \sqrt{\log \frac{1}{\eps}}$, there is separation in sample complexity: $\Theta(\frac{\sqrt{d}}{\alpha^2})$ v.s. $\tilde{\Theta} \left( \max \left( \frac{\sqrt{d}}{\alpha^2}, \frac{d
\varepsilon^2}{\alpha^2} \right) \right)$ between the truncation model and strong contamination model.

\section{Notation and Preliminaries}

\paragraph{Notation.} We denote the inner product of $\tmmathbf{x}, \tmmathbf{y} \in \R^d$ by $\langle \tmmathbf{x}, \tmmathbf{y}\rangle$. The identity matrix in $d$-dimensions is represented by $\mb{I}_d$. 
Let $\eps$ represent the mass of the truncation set, meaning that the mass of the observed part is $1 - \eps$, and let $\alpha$ denote the accuracy parameter. When there exists an absolute constant $c \in \mathbb{R}$ such that $A \leq c \cdot B$, we denote as $A \lesssim B$. Bold font is used to represent multivariate variables: e.g., $\tmmathbf{x,X}$.

The {\tmem{Mahalanobis distance}} between two vectors $x, y$ given
$\tmmathbf{\Sigma}$ is defined as,
\[ \|\tmmathbf{x}-\tmmathbf{y}\|_{\tmmathbf{\Sigma}} = \sqrt{(\tmmathbf{x}-\tmmathbf{y})^T \tmmathbf{\Sigma}^{- 1}
   (\tmmathbf{x}-\tmmathbf{y})} . \]

For a matrix $A \in \mathbb{R}^{m \times n}$ with entries $a_{i j}$, the
\emph{Frobenius norm} is defined as:
\[ \| A \|_F = \sqrt{\sum_{i = 1}^m \sum_{j = 1}^n | a_{i j} |^2} . \]
   
\paragraph{Truncated Gaussian Distribution.}
Let $\mathcal{N} (\tmmathbf{\mu}, \tmmathbf{\Sigma})$ represent the normal
distribution with mean $\tmmathbf{\mu}$ and covariance $\tmmathbf{\Sigma}$, whose probability density function is given by:
\[
    \mathcal{N} (\tmmathbf{\mu}, \tmmathbf{\Sigma}; \tmmathbf{x}) = \frac{1}{\sqrt{\det (2 \pi \tmmathbf{\Sigma})}} \exp \left( - \frac{1}{2}  \| \tmmathbf{x} -\tmmathbf{\mu}\|_{\tmmathbf{\Sigma}}^2 \right).
\] 
We denote the truncated normal distribution
restricted to a set $S$ as $\mathcal{N} (\tmmathbf{\mu}, \tmmathbf{\Sigma}, S)$, with the probability mass of $S$ under this distribution written as 
$\mathcal{N}(\tmmathbf{\mu}, \tmmathbf{\Sigma}, S)$. The corresponding probability density function is:
\[ \mathcal{N} (\tmmathbf{\mu}, \tmmathbf{\Sigma}, S ; \tmmathbf{x}) =
   \left\{\begin{array}{ll}
     \frac{1}{\mathcal{N} (\tmmathbf{\mu}, \tmmathbf{\Sigma}, S)} \cdot
     \mathcal{N} (\tmmathbf{\mu}, \tmmathbf{\Sigma}; \tmmathbf{x}) &
     \tmmathbf{x} \in S\\
     0 & \tmmathbf{x} \nin S
   \end{array}\right. . 
\]
We can then write the population negative log-likelihood $\bar{\ell}
(\cdot)$ for data coming from a truncated normal with mean $\tmmathbf{\mu}$ and covariance
matrix $\tmmathbf{I}_d$ as:
\begin{align}
  \bar{\ell} (\tmmathbf{v}) & =\mathbb{E}_{\tmmathbf{x} \sim \mathcal{N}
  (\tmmathbf{\mu}, \tmmathbf{I}_d, S)}  \left[ \frac{1}{2} \tmmathbf{x}^T
  \tmmathbf{x}-\tmmathbf{v}^T \tmmathbf{x} \right] \nonumber\\
  & \qquad + \log \left( \int_S \exp \left( - \frac{1}{2} \tmmathbf{z}^T
  \tmmathbf{z}+\tmmathbf{v}^T \tmmathbf{z} \right) d\tmmathbf{z} \right) .
  \label{eq:likelihood}
\end{align}

In this work, we focus on spherical Gaussian (covariance matrix is
$\tmmathbf{I}_d$), so our log-likelihood function only has one parameter. We can write the gradient of the negative log-likelihood function $\bar{\ell}$ as with respect to $\tmmathbf{v}$ ($\nabla \bar{l} (\tmmathbf{v})$) as follows:
\begin{align}
  \frac{\partial \bar{l} (\tmmathbf{v})}{\partial \tmmathbf{v}} 
   & = -\mathbb{E}_{\tmmathbf{x} \sim \mathcal{N} (\tmmathbf{\mu},
  \tmmathbf{I}_d, S)} [\tmmathbf{x}] +\mathbb{E}_{\tmmathbf{z} \sim
  \mathcal{N} (\tmmathbf{v}, \tmmathbf{I}_d, S)} [\tmmathbf{z}]
  \label{eq:gradient_of_negative_loglikelihood} 
\end{align}

Throughout this paper, we will use $S$ to indicate the support of $P$ after truncation and $\tmmathbf{\mu}_S = \mathbb{E}_{\tmmathbf{x} \sim \mathcal{N}(\tmmathbf{\mu}, \tmmathbf{I}_d, S)}[\tmmathbf{x}]$ the truncated mean of some multivariate normal $P$ (or the mean under truncation).

We will require the following two results from previous work:
\begin{lemma}[Strong convexity with truncation adapted {\cite[Lemma
4]{daskalakis2018efficient}}]
  \label{lemma:strong_convexity_negative_loglikelihood}
  Let $\tmmathbf{H}_{\ell}$ be the Hessian of the negative log likelihood
  function $\bar{\ell} (\tmmathbf{v})$, with the presence of arbitrary
  truncation $S$ such that $\mathcal{N} (\tmmathbf{\mu}, \tmmathbf{I}_d , S)
  \geqslant \beta$ for some $\beta \in (0, 1]$. Then it holds that
  \[ \tmmathbf{H}_{\ell} (\tmmathbf{v}) \succeq \frac{1}{2^{13}} \left(
     \frac{\beta}{C} \right)^4 \cdot \min \left\{ \frac{1}{4},
     \frac{1}{16\|\tmmathbf{\mu}\|_2^2 + 1} \right\} \cdot \tmmathbf{I}_d, \]
  where $C$ is a universal constant.
\end{lemma}

Let $\mathcal{D} = \{ P_v | v \in \mathbb{S}_d \}$ denote a family of distributions constructed in the following manner: Fix a one dimensional distribution $A$ and pick a unit $d$-dimensional vector $v \in \mathbb{S}_d$ uniformly at random. $P_v$ is a copy of $A$ in the direction of $v$ and standard normal in directions orthogonal to $v$.
\begin{proposition}[Sample complexity lower bound for high-dimensional testing {\cite[Proposition 7.1]{DBLP:journals/corr/DiakonikolasKS16c}}]
\label{prop:lbtesting}
Let $A$ be a distribution on $\R$ such that $A$ has mean 0 and $\chi^2(A, \mathcal{N}(0, 1))$ is finite. Then, there is no algorithm that, for any $d$, given $N < d/(8 \chi^2(A, \mathcal{N}(0, 1)))$ samples from a distribution $D$ over $\mathbb{R}^n$ which is either $\mathcal{N}(0, \tmmathbf{I}_d)$ or $P_v \in \mathcal{D}$, correctly distinguihes between the two cases with probability $2 / 3$.
\end{proposition}

\section{Testing under Unknown Truncation}
\label{sec:test_with_unknown_truncation}

When the truncation set is unknown, we will focus on three possible regimes depending on the relation between the accuracy and truncation parameter:

\begin{itemize}
    \item {
        $\eps \cdot \sqrt{\log (1 / \eps) } \lesssim \alpha$: in this case, the truncation size is much smaller than the required accuracy, meaning the change in the empirical mean after truncation is negligible (at most $\eps \cdot \sqrt{\log (1 / \eps}$). Therefore, applying the standard mean tester \cite{diakonikolas2022gaussian} \cref{alg:gmttester} with a sample complexity of $\cO(\sqrt{d}/{\alpha^2})$ is sufficient.
    }
    \item {
        $\eps \lesssim \alpha \lesssim \eps \cdot \sqrt{\log (1 / \eps)}$: Here, the truncation size is close to the accuracy threshold. An adversarial truncation (knowing the true mean) can select a truncation set that shifts the truncated mean by at least $\Omega(\eps \cdot \sqrt{\log (1 / \eps)})$. In this regime, we establish a lower bound of $\Omega(d/\eps)$, indicating a transition in sample complexity from $\Theta(\sqrt{d})$ to $\Theta(d)$.
    }
    \item {
        $\alpha \lesssim \eps$: When the truncation size exceeds the accuracy threshold, it has been shown that testing becomes information theoretically unfeasible \cite[Lemma 12]{daskalakis2018efficient} to produce an estimate that is closer than a constant in total variation distance to the true distribution even for single-dimensional truncated Gaussians.
    }
\end{itemize}
Our contribution are in the first two regimes and we will elaborate on in the following subsections.

\subsection{When Truncation Size is Much Smaller Than Accuracy $\eps \sqrt{\log 1/ \eps} \lesssim \alpha$}

In this subsection, we present \cref{mean_testing_small_truncation}. Given that the change in the expectation after truncation is minimal, it is sufficient to bound the change in both the mean and variance of the truncated normal distribution (as outlined in \Cref{lm:truncatedvsnontruncated}). We then apply the tester and analysis from \cite[Theorem 1.1]{diakonikolas2022gaussian}.
As a result, it is sufficient to apply the standard mean tester in \cref{alg:gmttester} with a sample complexity of $\cO(\sqrt{d}/{\alpha^2})$.
\begin{restatable}{theorem}{smalltruncation}
\label{mean_testing_small_truncation}
    There exists an algorithm (\Cref{alg:gmttester}) that, given i.i.d. samples from truncated Gaussian distribution $P$ with an unknown support set $S \subset \mathbb{R}^d$, can distinguish the following two cases based on the truncation mass parameter $\eps \in (0, 1)$ and the accuracy parameter $\alpha > 0$:
    \begin{itemize}
      \item \tmtextbf{(Completeness)} If $P$ is a truncated Gaussian distribution
      $\cN (\tmmathbf{0}, \mb{I}_d, S)$ and the truncation mass satisfies $1 - \cN (0, \mb{I}_d , S) \leqslant
      \eps$, the algorithm will output ''ACCEPT'' with probability at least $2 / 3$.
      
      \item \tmtextbf{(Soundness)} If $P$ is a truncated Gaussian distribution $\cN
      (\tmmathbf{\mu}, \mb{I}_d, S)$ where $\| \mb{\mu} \|_2 \geqslant \alpha \geqslant
      c_1 \cdot \eps \sqrt{\log \frac{1}{\eps}}$ for some
      constant $c_1 > 0$ and the truncation mass satisfies $1 - \cN (0, \mb{I}_d , S) \leqslant \eps$, the algorithm will output ''REJECT'' with probability at least $2 / 3$.
    \end{itemize}
    The algorithm requires $\cO \left( \frac{\sqrt{d}}{\alpha^2} \right)$ samples from $P$.
\end{restatable}

\begin{algorithm}[ht]  
\caption{GaussianMeanTester \cite{diakonikolas2022gaussian}} 
\label{alg:gmttester}
    \begin{algorithmic}[1]
      \Statex \textbf{Input}: 
      Sample access to distribution $P$ on $\R^d$ and $\alpha > 0$.
      \Statex \textbf{Output}: 
      ''ACCEPT'' if $P = \cN(\mathbf{0},\tmmathbf{I}_d,S)$, ``REJECT'' if $P = \cN(\tmmathbf{\mu},\tmmathbf{I}_d,S)$ and $\|\tmmathbf{\mu}\|_2 \geq \alpha$; both with probability at least $2/3$.
      \State Set $n = \cO ( \sqrt{d} / \alpha^2 )$.
      \State Sample $2n$ i.i.d.\ points from $p$ and denote them by $X_1,\dots,X_n$ and $Y_1,\dots,Y_n$.
      \State Define $Z = (1/n^2) (\sum_{i=1}^n X_i)^\top (\sum_{i=1}^n Y_i)$.
       \If{$ |Z| \leq O ( \sqrt{d}/n ) $}
           \State \Return ''ACCEPT''
       \Else \State \Return ''REJECT''
       \EndIf   
    \end{algorithmic}  
\end{algorithm} 
We now provide the proof sketch of \Cref{mean_testing_small_truncation}.
Given $2n$ i.i.d. samples from a $d$-variate truncated normal $P \sim \cN(\tmmathbf{\mu}, \mb{I}_d, S)$, let the sample set be $\{ \tmmathbf{x^{(1)}}, \dots,  \tmmathbf{x^{(n)} , \tmmathbf{y^{(1)}}, \dots, \tmmathbf{y^{(n)}}\}$, where 
$\mb{X} = \{x^{(1)}}, \dots, \tmmathbf{x^{(n)}}\}$, $\mb{Y} = \{\tmmathbf{y^{(1)}}, \dots, \tmmathbf{y^{(n)}}\}$.
The measure of $S$ under the non-truncated distribution $\cN(\tmmathbf{\mu}, \mb{I}_d)$ is at least $1- \eps$, where $0 \leq \eps < 1$. Define the empirical means of the sample sets in $\R^d$ as
\[
\bar{\mb{X}} \assign \frac{1}{n}  \sum_{i = 1}^n \mb{X}_i, \qquad \bar{\mb{Y}} \assign \frac{1}{n}  \sum_{i = 1}^n \mb{Y}_i.
\]
Our core test statistic is the inner product of these two empirical means: 
\begin{align}
\label{eq:z}
    Z = \langle \bar{\mb{X}}, \bar{\mb{Y}} \rangle 
\end{align}
Let $\tmmathbf{\mu}_S =\mathbb{E}_{\tmmathbf{x} \sim \mathcal{N}
(\tmmathbf{\mu}, \tmmathbf{I}_d, S)} [\tmmathbf{x}]$ denote the mean of the
truncated distribution, and let $\tmmathbf{\Sigma}_S =\mathbb{E}_{\tmmathbf{x} \sim
\mathcal{N} (\tmmathbf{\mu}, \tmmathbf{I}_d, S)}
[(\tmmathbf{x}-\tmmathbf{\mu}_S) \cdot (\tmmathbf{x}-\tmmathbf{\mu}_S)^T]$ be
the covariance matrix under truncation.

\begin{restatable}{lemma}{meanandvarianceofz}
\label{lm:z}
For the random variable $Z$ defined in \cref{eq:z}, obtained from two independent sets of $n$ samples (i.e. $2 n$ total samples) from $P$, the following holds: 
\begin{align}
    \mathbb{E}[{Z}] &= \langle \mathbb{E}[\bar{\mb{X}}], \mathbb{E}[{\bar{\mb{Y}}}] \rangle = \| \tmmathbf{\mu}_S \|_2^2 \label{eq:expz}\\
    \var[{Z}] & \leqslant \frac{\| \tmmathbf{\Sigma}_S \|_F^2}{n^2} + \frac{2}{n}  \| \tmmathbf{\Sigma}_S \|_F  \| \tmmathbf{\mu}_S
  \|_2^2 \label{eq:variancez}
\end{align}
\end{restatable}

\begin{restatable}[Truncated vs non-truncated parameters]{lemma}{truncatedvsnontruncated}
\label{lm:truncatedvsnontruncated}
Let $\tmmathbf{\mu}_S, \tmmathbf{\Sigma}_S$ be the mean and covariance of the truncated
Gaussian $\cN (\tmmathbf{\mu}, \mb{I}_d, S)$ with a measure of at least $1 - \eps$. Then the following holds:
\[ \| \tmmathbf{\mu}_S - \tmmathbf{\mu} \|_2 \leq \mathcal{O} (\varepsilon \cdot
   \sqrt{\log (1 / \varepsilon)})  \text{ and } \| \tmmathbf{\Sigma}_S -\tmmathbf{I}_d\|_F
   \leqslant \cO (\sqrt{d}) . \]
\end{restatable}
Using \cref{lm:z} and \cref{lm:truncatedvsnontruncated}, we compute the expectation and variance of $Z$. In the completeness case, the quantity $|Z - \|\tmmathbf{\mu} \||_2^2$ is small, with $\mathbb{E}[Z] < \cO(\alpha^2)$ and $\var{[Z]} \lesssim \alpha^4 $. In the soundness case, We can lower bound the expectation of $\tmmathbf{\mu}_S$ for $\cN (\tmmathbf{\mu}, \tmmathbf{I_d}, S)$, where $\| \tmmathbf{\mu} \|_2 \geqslant \alpha$, and show that $\mathbb{E} [Z] \geqslant \Omega (\alpha^2)$, and $\var[Z] \lesssim \mathbb{E}^2 [Z]$. This provides a clear separation between the two cases.

\subsection{When Truncation Size is Near Accuracy $\eps \lesssim \alpha \lesssim \eps \sqrt{\log 1/ \eps}$ }
As the truncation mass $\eps$ approaches to $\alpha$, the null and alternative hypothesis may overlap due to the non-negligible truncation size. This overlap occurs because it becomes possible to choose truncation regions that can substantially alter $\tmmathbf{\mu}_S$ by an amount comparable to $\alpha$, rendering the standard algorithm ineffective. Surprisingly, it presents a much greater challenge for our testing problem, where the sample complexity escalates to $\Omega(d)$, matching that of the existing robust learning algorithms \cite[Proposition 1.20]{diakonikolas2023algorithmic}.
\begin{restatable}{theorem}{samplecomplexitygap}
  \label{theorem:sample_complexity}
  The sample complexity for truncated mean testing when $\varepsilon \lesssim
  \alpha \lesssim \varepsilon \cdot \sqrt{\log \frac{1}{\varepsilon}}$ is $\Theta(d)$.
\end{restatable}
\paragraph{Mean Testing Lower Bound}
We now show the main idea of our lower bound proof.
Intuitively, the hard instance constructed in \Cref{lemma:lower_bound_truncation_near_accuracy} does exactly this: it modifies the mean by $\alpha$ and selects a random unit vector $v$ to define its direction in $\mathbb{R}^d$, thereby forming a $d$-variate truncated normal distribution in the soundness case. This family of hard instances will be difficult to distinguish from $\mathcal{N}(0, \tmmathbf{I}_d)$, the standard multivariate normal distribution without truncation. We can establish a $\Omega(d)$ sample complexity bound using lower bound machinery developed in \cite[Proposition 7.1]{DBLP:journals/corr/DiakonikolasKS16c}. This indicates that any tester will require a sufficient number of samples to estimate the hidden direction $v$ before being able to differentiate between the null and alternative hypothesis.

\begin{restatable}[Lower Bound for Mean Testing with Unknown Truncation When \( \eps \lesssim \alpha \lesssim \eps \sqrt{\log (1/\eps)} \)]{lemma}{lbunknowntruncation}
\label{lemma:lower_bound_truncation_near_accuracy}
    No algorithm can distinguish between \(\mathcal{N}(\mathbf{0}, \mathbf{I}_d)\) and a family of truncated normal distribution of the form: \(\mathcal{N}(\tmmathbf{v}, \mathbf{I}_d, S)\) with measure \(\eps\) on the truncation set \(\bar{S} = \mathbb{R}^d \backslash S\), for any \(\eps < 1\) and some \(\| \tmmathbf{v} \|_2 = \alpha = \Theta (\eps \sqrt{\log (1/\eps)})\), using fewer than \(\Omega\left(d/\eps \right)\) samples with a probability greater than \(2/3\).
\end{restatable}

The complete proof is provided in \cref{appendix:sec:chi_square_upper_bound}. Below, we present a sketch of the proof for \cref{lemma:lower_bound_truncation_near_accuracy}.
We begin by constructing a one-dimensional truncated normal distribution $A = \mathcal{N} (\alpha, 1, S)$, where the truncated mass is $\eps$. This means 
$\Pr_{x \sim \mathcal{N} (\alpha, 1)} [x \in S] = 1 - \eps$. We can determine the $1 - \eps$ quantile as:
 \[ 
     b = \alpha + \sqrt{2} \tmop{erf}^{- 1} (1 - 2 \eps) . 
 \]
which defines the truncation set as 
$S = (- \infty, b]$.

Let $\alpha (\eps) = \alpha = \Theta \left( \eps \sqrt{\log \frac{1}{\eps}} \right)$. For any $\eps$, we can find a constant $c_2 = \Theta(1)$ such that $\mathbb{E} [A] = 0$:
\[
    \mathbb{E}_{X \sim A} [X] 
    =
    \alpha - \frac{\exp \left( - \frac{1}{2} 
    \left( b - \alpha \right)^2 \right)}{\sqrt{2 \pi}  (1 -
    \varepsilon)} 
    = 0, 
\]
which is equivalent to:
\[ 
    \frac{\exp (- (\tmop{erf}^{- 1} (1 - 2 \eps))^2)}{\sqrt{2 \pi}  (1 - \eps)} 
    =
    \Theta \left( \eps \sqrt{\log \frac{1}{\eps}} \right) 
    = \alpha. 
\]
Next, we compute an upper bound on the chi-squared divergence between the truncated distribution $A$ and the standard normal distribution $\mathcal{N}(0, 1)$. We find that 
\[
 \chi^2 (A, \mathcal{N} (0, 1)) \leq \cO(\eps + \alpha^2),
\]
We now apply \cref{prop:lbtesting} \cite[Proposition 7.1]{DBLP:journals/corr/DiakonikolasKS16c}, and obtain a lower bound of
\[ 
    \Omega \left( \frac{d}{\eps + \alpha^2} \right) 
    =
    \Omega \left(\frac{d}{\eps} \right). 
\]

\paragraph{Mean Testing Upper Bound}
We apply the standard learning-to-test approach: first we estimate the pre-truncation mean of the truncated normal using $\cO(d / \alpha^2)$ samples, following \cite[Proposition 1.20]{diakonikolas2023algorithmic}. This gives an estimate $\hat{\mathbf{\mu}}$ that is within $\alpha$ of the true mean before truncation. If $\hat{\mathbf{\mu}}$ is sufficiently close to zero, we return "ACCEPT". Otherwise, return "REJECT".

\section{Testing under known truncation}
In this section, we demonstrate in \cref{theorem:known_truncation_mean_testing} that when the truncation set is known, an alternative yet straightforward algorithm, which leverages the gradient of the maximum likelihood estimator, achieves the optimal sample complexity of $\cO(\sqrt{d})$ across all parameter regimes. 
As a result, it is sufficient to apply  \cref{alg:gmttesterknowntruncation} with a sample complexity of $\cO(\sqrt{d}/{\alpha^2})$.

\begin{algorithm}[ht]
\caption{GaussianMeanTester with known truncation} 
\label{alg:gmttesterknowntruncation}
    \begin{algorithmic}[1]
      \Statex \textbf{Input}: 
      Sample access to the truncated normal $P$ on $\R^d$ and $\alpha > 0$ and oracle access to its support set $S$.
      \Statex \textbf{Output}: 
      ``ACCEPT'' if $P = \cN(\mathbf{0}, \tmmathbf{I}_d, S)$, ``REJECT'' if $P = \cN(\tmmathbf{\mu}, \tmmathbf{I}_d, S)$ and $\| \tmmathbf{\mu} \|_2 \geq \alpha$; both with probability at least $2/3$.
      \State Compute $\tmmathbf{\mu}'_S = \mathbb{E}_{\tmmathbf{x} \sim \mathcal{N}(\mathbf{0}, \tmmathbf{I}_d, S)}[\tmmathbf{x}]$ .
      \State Set $n = \cO ( \sqrt{d} / \alpha^2 )$.
      \State Sample $2n$ i.i.d.\ points from $P$ and denote them by $X_1,\dots,X_n$ and $Y_1,\dots,Y_n$.
      \State $Z_1 = \left( \frac{1}{n} \sum_{i=1}^n X_i - \tmmathbf{\mu}'_S \right)^\top \left(\frac{1}{n} \sum_{i=1}^n Y_i - \tmmathbf{\mu}'_S \right)$.
       \If{$ |Z_1| \leq O ( \alpha^2 ) $}
           \State \Return ''ACCEPT''
       \Else \State \Return ''REJECT''
       \EndIf   
    \end{algorithmic}  
\end{algorithm} 
The algorithm works as follows: Given the support $S$, it first calculates the truncated mean for the standard multivariate normal, denoted as $\mathbf{\mu}'_S$. Next, it draws $2n$ i.i.d. samples from the truncated normal distribution $P$ with unknown mean. The algorithm then computes the statistic: 
$$
Z_1 = \left( \frac{1}{n} \sum_{i=1}^n X_i - \tmmathbf{\mu}'_S \right)^\top \left(\frac{1}{n} \sum_{i=1}^n Y_i - \tmmathbf{\mu}'_S \right).
$$
The algorithm will return ''ACCEPT'' if $|Z_1| \leq \cO(\alpha^2)$ and ''REJECT'' otherwise.

The proof of \Cref{theorem:known_truncation_mean_testing} relies on the following two lemmas.

\begin{restatable}{lemma}{meanvariancezknowntrunc}
\label{lemma:expectation_variance_with_known_truncation}
  Let $Z_1$ be the statistics in \cref{alg:gmttesterknowntruncation} Line 4, and $\tmmathbf{\mu}'_S
  =\mathbb{E}_{\tmmathbf{x} \sim \mathcal{N} (\tmmathbf{0}, \tmmathbf{I}_d, S)} [\tmmathbf{x}]$
  (truncated mean under zero mean). Let $\tmmathbf{\mu}_S$ be the truncated mean of
  the unknown Gaussian $P$, we can show that
  \[ \mathbb{E} [Z_1] = \| \tmmathbf{\mu}_S -\tmmathbf{\mu}_S' \|_2^2 . \]
  \[ \tmop{Var} [Z_1] \leqslant O (\alpha^4 + \alpha^2 \cdot \|
     \tmmathbf{\mu}_S -\tmmathbf{\mu}_S' \|_2^2) . \]
\end{restatable}

\begin{restatable}[Gap of Mean under Truncation]{lemma}{gapmeanknwontrunc}
\label{lemma:gap_truncated_mean}
Let \(\mathbb{E}_{\tmmathbf{y} \sim \mathcal{N} (\tmmathbf{0}, \tmmathbf{I}_d, S)} [\tmmathbf{y}] = \tmmathbf{\mu}_S'\) and \(\mathbb{E}_{\tmmathbf{y} \sim \mathcal{N} (\tmmathbf{\mu}'', \tmmathbf{I}_d, S)} [\tmmathbf{x}] = \tmmathbf{\mu}_S''\), where \(\|\tmmathbf{\mu}''\|_2^2 \geq \alpha^2\). Additionally, assume that \(\mathcal{N} (\tmmathbf{\mu}'', \tmmathbf{I}_d, S) \geq 1 - \beta\) for some constant \(\beta\). Then, it holds that
\[
\|\tmmathbf{\mu}_S' - \tmmathbf{\mu}_S''\|_2^2 \geq \Omega(\alpha^2).
\]
\end{restatable}

\begin{proof}[Proof sketch]
Consider the negative log-likelihood function, $\bar{\ell}(\tmmathbf{0})$, with the mean set to $\tmmathbf{0}$ as the input parameter. This function is defined for a population drawn from a truncated normal distribution with an unknown mean ${\tmmathbf{\mu}}$. From \eqref{eq:gradient_of_negative_loglikelihood}, we can express the gradient of the negative log-likelihood with respect to the mean evaluated at $\tmmathbf{0}$, as follows:
    \[ \nabla \bar{\ell} (\tmmathbf{0}) = -\mathbb{E}_{\tmmathbf{x} \sim \mathcal{N}
   (\tmmathbf{\mu}, \tmmathbf{I}_d, S)} [\tmmathbf{x}]
   +\mathbb{E}_{\tmmathbf{z} \sim \mathcal{N} (\tmmathbf{0}, \tmmathbf{I}_d;
   S)} [\tmmathbf{z}] =\tmmathbf{\mu}_S -\tmmathbf{\mu}_S' . \]
   Likewise, when evaluating the gradient at $\tmmathbf{\mu}$, we have
   \[ \nabla \bar{\ell} (\tmmathbf{\mu}) = -\mathbb{E}_{\tmmathbf{x} \sim \mathcal{N}
   (\tmmathbf{\mu}, \tmmathbf{I}_d, S)} [\tmmathbf{x}]
   +\mathbb{E}_{\tmmathbf{z} \sim \mathcal{N} (\tmmathbf{\mu}, \tmmathbf{I}_d;
   S)} [\tmmathbf{z}] = \tmmathbf{0} . \]
   So, $\nabla\bar{\ell}(\tmmathbf{0})$ represents the difference between the truncated mean of the underlying distribution and that of the distribution with mean $\tmmathbf{0}$. From \Cref{lemma:strong_convexity_negative_loglikelihood}, we know that $\bar{\ell}(\cdot)$ is $\lambda_0$-strongly convex, and $\lambda_0$ is a constant if $\beta$ is a constant.
   Therefore, by leveraging the properties of strong convexity and applying the Cauchy–Schwarz inequality, we obtain the following result:
    \begin{align}
      \MoveEqLeft \sqrt{\|\tmmathbf{\mu}-\tmmathbf{0}\|_2^2 \cdot \| \nabla \bar{l}
      (\tmmathbf{\mu}) - \nabla \bar{l} (\tmmathbf{0})\|_2^2} \nonumber
       \geqslant \langle \nabla \bar{l} (\tmmathbf{\mu}) - \nabla \bar{l}
      (\tmmathbf{0}), \tmmathbf{\mu}-\tmmathbf{0} \rangle \geqslant
      \frac{\lambda_0}{2}  \|\tmmathbf{\mu}\|_2^2  \nonumber
    \end{align}
By simplifying the expression and substituting $\tmmathbf{\mu}$ with any $\| \tmmathbf{\mu''} \|_2^2 \geqslant \alpha^2$, we can show that:
\[
    \| \tmmathbf{\mu}_S'' -\tmmathbf{\mu}_S' \|_2^2 \geqslant \Omega (\alpha^2).
\]  
\end{proof}

\begin{restatable}[Known truncation tester]{theorem}{knownTruncationAlgoTheorem}
\label{theorem:known_truncation_mean_testing}
There exists an algorithm (\Cref{alg:gmttesterknowntruncation}) that takes i.i.d. samples from truncated normal
Gaussian $P$ and given oracle access to $S \subset \mathbb{R}^d$, the
effective support of $P$, distinguishing the cases for parameters (mass of
truncation) $0 < \varepsilon \leqslant 1 - \beta$, where $\beta$ is a constant and (accuracy) $\frac{1}{4} \geqslant \alpha > 0$:
\begin{itemize}
  \item \tmtextbf{(Completeness)} $P$ is a truncated Gaussian distribution
  $\cN (\tmmathbf{0}, \mb{I}_d, S)$ and $1 - \cN (\tmmathbf{0}, \mb{I}_d , S) \leqslant
  \varepsilon$. In this case, the algorithm will output yes with probability
  at least $2 / 3$.
  
  \item \tmtextbf{(Soundness)} $P$ is a truncated Gaussian distribution $\cN
  (\tmmathbf{\mu}, \mb{I}_d, S)$ where $\| \tmmathbf{\mu} \|_2 \geqslant \alpha$ and $1 -
  \cN (\tmmathbf{0}, \mb{I}_d , S) \leqslant \varepsilon$. In this case, the algorithm
  will output no with probability at least $2 / 3$.
\end{itemize}
The algorithm will take $\cO \left( \frac{\sqrt{d}}{\alpha^2} \right)$ samples
from $P$.
\end{restatable}
 
\begin{proof}[Proof sketch]
    Using \Cref{lemma:expectation_variance_with_known_truncation} and \Cref{lemma:gap_truncated_mean}, we apply Chebyshev inequality in the two cases:
\begin{enumerate}
    \item \textbf{Completeness}: We know that $\mathbb{E}[Z_1] = 0$ and $\var[Z_1] \leq O(\alpha^4)$. Thus, by Chebyshev's inequality, with probability at least $2/3$ using, 
    \[
       Z_1 \leq \cO(\alpha^2).
    \]
    \item \textbf{Soundness}: Let the non-truncated mean be $\tmmathbf{\mu}''$ (and $\| \tmmathbf{\mu}'' \|^2_2 \geq \alpha^2$) with $\| \tmmathbf{\mu}'' \|^2_2 \geq \alpha^2$. Here, $\mathbb{E}[Z_1] = \| \tmmathbf{\mu}''_S -  \tmmathbf{\mu}'_S \|_2^2$ and $\var[Z_1] \leq O(\alpha^4 + \alpha^2 \| \tmmathbf{\mu}''_S -  \tmmathbf{\mu}'_S \|_2^2)$. Applying Chebyshev's inequality, with probability at least $2/3$, we have 
    \[
        Z_1 \geq \| \tmmathbf{\mu}''_S -  \tmmathbf{\mu}'_S \|_2^2 - O(\alpha^2 + \alpha \| \tmmathbf{\mu}''_S -  \tmmathbf{\mu}'_S \|) \geq \Omega(\alpha^2).
    \]
\end{enumerate}    
\end{proof}

\section{Conclusion and Future Work}
In this work, we highlight the critical interplay between truncation mass $\eps$ and accuracy $\alpha$ in determining the sample complexity required for {\GMT} in both known and unknown truncation regimes.
\begin{itemize}
    \item \textbf{Unknown Truncation}: For $\eps \lesssim \alpha/ \sqrt{\log(1/\alpha)}$, we establish the tight sample complexity of $\Theta(\sqrt{d})$, indicating the effectiveness of testing under mild truncation. However, as $\eps$ approaches $\alpha$, the sample complexity sharply jumps to $\Theta(d)$,  indicating a much more challenging testing regime (where testing brings no sample complexity savings over learning). Furthermore, when $\alpha \leq \eps$, the sample complexity becomes infinite, as testing becomes information-theoretically unfeasible.
    \item \textbf{Known Truncation}: In contrast, when the truncation is known, the sample complexity remains $\Theta(\sqrt{d})$ across all parameter ranges, even when $\eps \gtrsim \alpha$. Thus, having prior knowledge of truncation can facilitate efficient testing regardless of the relationship between $\alpha$ and $\eps$.
\end{itemize}
Overall, this is the first work that provide valuable insights into the sample complexity for efficient {\GMT}, emphasizing the importance of understanding truncation in designing algorithms for robust statistics.

In future work, we aim to generalize the soundness case by extending our analysis to any arbitrary (unknown) covariance matrix $\tmmathbf{\Sigma}$, beyond the identity-covariance case. Another avenue of research, inspired by the recent line of work on convex truncation~\cite{DeNS23}, is to explore whether structural assumptions on the truncation set (whether known or unknown), for instance convexity or rotational symmetry, could enable significantly more sample-efficient algorithms for the task.

\section{Acknowledgments}
 Cl\'{e}ment L. Canonne gratefully acknowledge support from an ARC DECRA (DE230101329).
 Joy Qiping Yang is supported by a JD Technology Research Scholarship in Artificial intelligence.
 Themis Gouleakis gratefully acknowledges support by the Nanyang Technology University start up grant. 
 Yuhao Wang sincerely thanks Yi-Jun Chang for the inspiration of this work through the final project of his CS6234 Advanced Algorithms class at National University of Singapore.
 


\newcommand{\etalchar}[1]{$^{#1}$}

\newpage
\onecolumn
\appendix

\section{Omitted proofs from \Cref{sec:test_with_unknown_truncation}}

We will require the below well-known results for the statistic $Z$, show in~\cite[Lemma~4.1]{canonne2021random}. For completeness, we provide the proof below:

\meanandvarianceofz*

\begin{proof}
We will prove this for any $d$-dimensional distribution $X \sim P$.
Suppose the $\tmmathbf{\mu} = \mathbb{E}[X]$ and denote $\Sigma$ its covariance matrix.
Draw $\mb{X} = \{x^{(1)}, \dots, x^{(n)}\}$, $\mb{Y} = \{y^{(1)}, \dots, y^{(n)}\}$ i.i.d. $2 n$ samples from $P$; let
$\bar{\tmmathbf{X}} \assign \frac{1}{n}  \sum_{i = 1}^n \tmmathbf{X}_i$,
$\bar{\tmmathbf{Y}} \assign \frac{1}{n}  \sum_{i = 1}^n \tmmathbf{Y}_i$.
\begin{align*}
\hspace{-1cm}
     Z =
     \langle \bar{\tmmathbf{X}}, \bar{\tmmathbf{Y}} \rangle 
      = \frac{1}{n^2} \sum_{i = 1}^d \sum_{k = 1}^n \sum_{l = 1}^n X_{l, i} Y_{k, i}. \qquad 
     \mathbb{E} [Z]
     = 
     \| \tmmathbf{\mu} \|_2^2. 
\end{align*}
The proof follows from the fact that $\bar{\tmmathbf{X}}$ and $\bar{\tmmathbf{Y}}$ are independent, thus $\tmmathbf{X}_i \indep \tmmathbf{X}_j$ regardless of $i$ and $j$. 
Note that $\tmop{Var} [Z] = \mathbb{E} [Z^2] -\mathbb{E}^2 [Z]$, we start by computing the second moment of the statistic:

{\begin{align*}
  \mathbb{E} [Z^2] 
  & = 
  \mathbb{E} \left[ \left( \frac{1}{n^2}  \sum_{i = 1}^d \sum_{k = 1}^n \sum_{l = 1}^n X_{l, i} Y_{k, i} \right)  
  \left( \frac{1}{n^2}  \sum_{j = 1}^d \sum_{k' = 1}^n \sum_{l' = 1}^n X_{l', j} Y_{k', j} \right) \right] \\
  & = 
  \frac{1}{n^4}  \sum_{i = 1}^d \sum_{j = 1}^d \sum_{k = 1}^n \sum_{l = 1}^n \sum_{k' = 1}^n \sum_{l' = 1}^n 
  \mathbb{E} [X_{l, i} Y_{k, i} X_{l', j} Y_{k', j}] \\
  & = 
  \frac{1}{n^4}  \sum_{i = 1}^d \sum_{j = 1}^d \left( \sum_{l = 1}^n \sum_{l' = 1}^n \mathbb{E} [X_{l, i} X_{l', j}] 
  \sum_{k = 1}^n \sum_{k' = 1}^n \mathbb{E} [Y_{k, i} Y_{k', j}] \right) \\
  & = 
  \frac{1}{n^4}  \sum_{i = 1}^d \sum_{j = 1}^d \left( \sum_{l = 1}^n \sum_{l' = 1}^n \mathbb{E} [X_{l, i} X_{l', j}] \right) 
  \left( \sum_{k = 1}^n \sum_{k' = 1}^n \mathbb{E} [Y_{k, i} Y_{k', j}] \right) \\
  & = 
  \frac{1}{n^4}  \sum_{i = 1}^d \sum_{j = 1}^d \left( \sum_{l = 1}^n \sum_{l' = 1}^n \mathbb{E} [X_{l, i} X_{l', j}] \right)^2 \\
  & = 
  \frac{1}{n^4}  \sum_{i = 1}^d \sum_{j = 1}^d \left( \sum_{l = 1}^n \mathbb{E} [X_{l, i} X_{l, j}] + 
  \sum_{l \neq l'} \mathbb{E} [X_{l, i}] \mathbb{E} [X_{l', j}] \right)^2 \\
  & = 
  \frac{1}{n^4}  \sum_{i = 1}^d \sum_{j = 1}^d \left( \sum_{l = 1}^n \tmop{Cov} (X_{l, i}, X_{l, j}) + 
  \mathbb{E} [X_{l, i}] \mathbb{E} [X_{l, j}] + n (n - 1) \mu_i \mu_j \right)^2 \\
  & = 
  \sum_{i = 1}^d \sum_{j = 1}^d \left( \frac{1}{n} \tmop{Cov} (X_i, X_j) + \mu_i \mu_j \right)^2 \\
  & = 
  \sum_{i = 1}^d \sum_{j = 1}^d \left( \frac{1}{n^2} \Sigma_{i, j}^2 + \frac{2}{n} \Sigma_{i, j} \mu_i \mu_j + \mu_i^2 \mu_j^2 \right) \\
  & = 
  \sum_{i = 1}^d \sum_{j = 1}^d \left( \frac{1}{n^2} \Sigma_{i, j}^2 + \frac{2}{n} \Sigma_{i, j} \mu_i \mu_j \right) 
  + \| \mu \|_2^4
\end{align*}}

Then substitute \cref{eq:expz} to complete the computation of \cref{eq:variancez}
\begin{align*}
    \tmop{Var} [Z] 
    =
    \mathbb{E} [Z^2] -\mathbb{E}^2 [Z] 
    & = \sum_{i = 1}^d \sum_{j
   = 1}^d \left( \frac{1}{n^2} \Sigma_{i, j}^2 + \frac{2}{n} \Sigma_{i, j} \mu_i \mu_j \right) \\
   & = 
   \frac{1}{n^2}  \sum_{1
   \leqslant i, j \leqslant d} \Sigma^2_{i, j} + \frac{2}{n}  \sum_{i, j}
   \Sigma_{i, j} \mu_i \mu_j \\
   & =
   \frac{\| \Sigma \|_F^2}{n^2} + \frac{2}{n}  \sum_{i, j} \Sigma_{i, j} \mu_i \mu_j \\
   & \leqslant 
   \frac{\| \Sigma \|_F^2}{n^2} + \frac{2}{n}  \sqrt{\sum_{i, j} \Sigma_{i, j}^2}  \sqrt{\sum_{i, j} \mu_i^2 \mu_j^2} \tag{By Cauchy-Schwarz}\\
   & = 
   \frac{\| \Sigma \|_F^2}{n^2} + \frac{2}{n}  \| \Sigma \|_F  \cdot \| \mu \|_2^2\qedhere
\end{align*}
\end{proof}

\truncatedvsnontruncated*
\begin{proof}

We establish each statement separately. 

\textbf{Bound on the mean $\| \tmmathbf{\mu}_S - \tmmathbf{\mu} \|_2$:}
Consider the region $\bar{S}$, which contributes the most to the change in the mean or covariance matrix in terms of the Frobenius norm. Let $\tmmathbf{v}_S$ be the unit vector in the direction of the truncated mean $\tmmathbf{\mu}_S$. For any unit vector $\tmmathbf{v}$, the region that impacts the expectation $\mathbb{E}[\tmmathbf{v}^T \tmmathbf{x}]$ or $\var[\tmmathbf{v}^T \tmmathbf{x}]$ the most corresponds to truncating the $\eps$-tail of $\tmmathbf{v}^T \tmmathbf{x}$. The change in the mean in this direction can be bounded by $\cO(\eps \sqrt{\log (1/\eps)})$, and similarly, the variance of $\tmmathbf{v}^T \tmmathbf{x}$ changes by at most 
$\cO(\eps \log (1/\eps))$, as can be shown by relatively standard and elementary computations on a single-dimensional standard Gaussian.
Thus, for the mean shift, we have
\[
    \| \tmmathbf{\mu}_S - \tmmathbf{\mu} \|_2 = \| \tmmathbf{\mu}_S \| = \cO(\eps \sqrt{\log (1/\eps)})
\]
Even if the region $\bar{S}$ fully truncates its $\eps$ mass in the direction of $\tmmathbf{v}_S$, the mean shift in that direction is at most $\cO(\eps \sqrt{\log (1/\eps)})$.

\textbf{Bound on the covariance $\| \tmmathbf{\Sigma}_S -\tmmathbf{I}_d\|_F$:}
Next, we turn to the covariance matrix. For any unit vector $\tmmathbf{v}$, the variance of $\tmmathbf{v}^T \tmmathbf{x}$ in the truncated distribution can be expressed as:
\[
    \var[\tmmathbf{v}^T \tmmathbf{x}]= \mathbb{E}[(\tmmathbf{v}^T\tmmathbf{x})^2] - 
    (\mathbb{E}[\tmmathbf{v}^T\tmmathbf{x}])^2.
\]
For the truncated Gaussian, the variance of $\tmmathbf{v}^T \tmmathbf{x}$ differs from 1 by at most $\cO(\eps\log (1/\eps))$, i.e.,\footnote{We believe one can prove a bound of $\cO(\eps \log 1/\eps)$ with a more sophisticated analysis; however, this weaker bound suffices for our purposes.},
\[ 
    \tmop{Var} [\tmmathbf{v}^T\tmmathbf{x}] - 1 =\mathbb{E} [(\tmmathbf{v}^T\tmmathbf{x})^2] -\mathbb{E} [\tmmathbf{v}^T\tmmathbf{x}]^2 - \tmmathbf{v}^T \tmmathbf{I}_d \tmmathbf{v}
\]
where
\begin{align*}
    \mathbb{E} [\tmmathbf{v}^T\tmmathbf{x}]^2
    & =
    (\tmmathbf{v}^T \tmmathbf{\mu}_S) \cdot (\tmmathbf{\mu}_S^T \tmmathbf{v}) \\
    \mathbb{E} [(\tmmathbf{v}^T\tmmathbf{x})  (\tmmathbf{v}^T\tmmathbf{x})^T] - \tmmathbf{v}^T \tmmathbf{I}_d \tmmathbf{v}
    & =
    \tmmathbf{v}^T (\mathbb{E} [\tmmathbf{x} \tmmathbf{x}^T] - \tmmathbf{I}_d) \tmmathbf{v} = \tmmathbf{v}^T  (\Sigma_S
   + \tmmathbf{\mu}_S \tmmathbf{\mu}^T_S - I_d) \tmmathbf{v} 
\end{align*}
Thus
\[ 
    | \tmop{Var} [\tmmathbf{v}^T\tmmathbf{x}] - 1 | = | \tmmathbf{v}^T  (\Sigma_S - I_d) \tmmathbf{v} | \leqslant \cO
   (\varepsilon \log 1 / \varepsilon) . 
\]
Now, recall the relationship between the spectral norm and the Frobenius norm: if the spectral norm of $\tmmathbf{\Sigma}_S - \tmmathbf{I}_d$ is bounded by $\cO(\eps \log(1/\eps)) = \cO(1)$, then the Frobenius norm satisfies
$\cO(\sqrt{d})$.

\end{proof}

\smalltruncation*
\begin{proof}
  When $\tmmathbf{X}, \tmmathbf{Y}$ come from the truncated Gaussian
  distribution $\cN (\tmmathbf{\mu}, \tmmathbf{I}_d, S)$, by 
  \cref{lm:z}, we know that for random variable $Z$
  \[ Z = \langle \bar{\tmmathbf{X}}, \bar{\tmmathbf{Y}} \rangle, \]
  the following condition holds:
  
  \begin{align*}
    \mathbb{E} [Z] & = \langle \mathbb{E}[\overline{\tmmathbf{X}}],
    \mathbb{E}[\overline{\tmmathbf{Y}}] \rangle = \|\tmmathbf{\mu}\|_2^2\\
    \var[Z] & \leqslant \frac{\| \Sigma_S \|_F^2}{n^2} + \frac{2}{n} \|
    \Sigma_S \|_F \|\tmmathbf{\mu}_S \|_2^2 .
  \end{align*}
  
  By \cref{lm:truncatedvsnontruncated},
  \begin{equation}
    \|\tmmathbf{\mu}_S -\tmmathbf{\mu}\|_{I_d} \leqslant \cO \left(
    \varepsilon \cdot \sqrt{\log (1 / \varepsilon)} \right) \infixand \|
    \Sigma_S -\tmmathbf{I}_d \|_F \leqslant \cO (\sqrt{d}) .
  \end{equation}
  \paragraph{Completeness:} In the completeness case, we want to show the
  expectation of $Z$ is small when $\|\tmmathbf{\mu}\|_2 = 0$.
  
  \begin{align*}
    \mathbb{E} [Z] & = \|\tmmathbf{\mu}_S \|_2^2 \leqslant O (\varepsilon^2
    \cdot \log (1 / \varepsilon)) \leqslant O (c_1^2 \cdot \alpha^2) = O
    (\alpha^2) .\\
    \tmop{Var} [Z] & \leqslant \frac{\| \Sigma_S \|_F^2}{n^2} + \frac{2}{n} \|
    \Sigma_S \|_F \|\tmmathbf{\mu}_S \|_2^2\\
    & \leqslant \frac{(\| \Sigma_S -\tmmathbf{I}_d \|_F +\|\tmmathbf{I}_d
    \|_F)^2}{n^2} + \frac{2}{n}  (\| \Sigma_S -\tmmathbf{I}_d \|_F
    +\|\tmmathbf{I}_d \|_F)  (\|\tmmathbf{\mu}_S -\tmmathbf{\mu}\|_2
    +\|\tmmathbf{\mu}\|_2)^2\\
    & \lesssim \frac{(\sqrt{d} +\|\tmmathbf{I}_d \|_F)^2}{n^2} + \frac{2}{n} 
    (\sqrt{d} +\|\tmmathbf{I}_d \|_F)  \left( \varepsilon \cdot \sqrt{\log (1
    / \varepsilon)} + 0 \right)^2\\
    & = O \left( \frac{d}{n^2} \right) + O \left( \frac{\sqrt{d}}{n}  \eps^2
    \log 1 / \varepsilon \right)\\
    & \lesssim \underbrace{\frac{d}{n^2}}_{\ll \alpha^4} +
    \underbrace{\frac{\sqrt{d}}{n}\cdot \alpha^2}_{\ll \alpha^4}\\
    & \lesssim \alpha^4
  \end{align*}
  
  Since $n \gtrsim \frac{\sqrt{d}}{\alpha^2}$ and $\alpha \gtrsim \eps\sqrt{\log(1/\eps)}$, both terms are much smaller
  than $\alpha^4$. By Chebyshev's inequality, we have:
  \[ \Pr \left[ Z -\|\tmmathbf{\mu}_S \|_2^2 \geqslant \frac{1}{2}
     \|\tmmathbf{\mu}_S \|_2^2 \right] \leqslant \frac{4 \tmop{Var}
     [Z]}{\mathbb{E}^2 [Z]} \]
  Using the bounds on $\mathbb{E} [Z]$ and $\var{[Z]}$, this gives:
  \[ \Pr [Z \geqslant \Omega (\alpha^2)] \leqslant \frac{1}{9} . \]
  Thus, the algorithm outputs ``ACCEPT'' with high probability in the
  completeness case.
  
  \paragraph{Soundness:} In the soundness case, we have that
  $\|\tmmathbf{\mu}\|_2 \geqslant c_1 \alpha \geqslant \varepsilon \cdot
  \sqrt{\log (1 / \varepsilon)}$. We now show that $Z$ is large in this case:
  \begin{equation}
    \mathbb{E} [Z] = \|\tmmathbf{\mu}_S \|_2^2 \geqslant (\|\tmmathbf{\mu}\|_2
    -\|\tmmathbf{\mu}-\tmmathbf{\mu}_S \|_2)^2 = \left( \|\tmmathbf{\mu}\|_2 -
    \cO \left( \eps \cdot \sqrt{\log \frac{1}{\varepsilon}} \right) \right)^2
    \geqslant \left( \|\tmmathbf{\mu}\|_2 - \cO (\alpha) \right)^2 \geqslant
    \Omega (\|\tmmathbf{\mu}\|_2^2) . \label{eq:truncated_mean_still_large}
  \end{equation}
  Similarly, the variance of $Z$ is bounded as:
  \begin{eqnarray*}
    \tmop{Var} [Z] & \leqslant & \frac{\| \Sigma_S \|_F^2}{n^2} + \frac{2}{n}
    \| \Sigma_S \|_F \|\tmmathbf{\mu}_S \|_2^2\\
    & \leqslant & \frac{(\| \Sigma_S -\tmmathbf{I}_d \|_F +\|\tmmathbf{I}_d
    \|_F)^2}{n^2} + \frac{2}{n}  (\| \Sigma_S -\tmmathbf{I}_d \|_F
    +\|\tmmathbf{I}_d \|_F)  \|\tmmathbf{\mu}_S \|_2^2\\
    & \lesssim & \frac{(\sqrt{d} + \sqrt{d})^2}{n^2} + \frac{2}{n}  (\sqrt{d}
    + \sqrt{d})  \|\tmmathbf{\mu}_S \|_2^2\\
    & = & O \left( \frac{d}{n^2} \right) + O \left(
    \frac{\sqrt{d}}{n} \|\tmmathbf{\mu}_S \|_2^2 \right)\\
    & \lesssim & \frac{d}{n^2} + \frac{\sqrt{d}}{n} \mathbb{E} [Z] \lesssim
    \mathbb{E} [Z]^2
  \end{eqnarray*}
  using that $n \geqslant \Omega \left( \frac{\sqrt{d}}{\alpha^2} \right)$ and recalling that $\mathbb{E} [Z] = \| \tmmathbf{\mu}_S \|_2^2 \geqslant \Omega
  (\alpha^2)$ via \eqref{eq:truncated_mean_still_large} in the last step. By
  Chebyshev's inequality, we get:
  \[ \Pr \left[ \|\tmmathbf{\mu}_S \|_2 - Z \geqslant \frac{1}{2}
     \mathbb{E}[Z] \right] \leqslant \frac{4 \tmop{Var} [Z]}{\mathbb{E}[Z]^2}
     \Rightarrow \Pr [Z \leqslant O (\|\tmmathbf{\mu}_S \|_2)] \leqslant
     \frac{1}{9} \]
  Thus, with high probability:
  \[ \Pr [Z \leqslant O (\alpha^2)] \leqslant \frac{1}{9} . \]
  Hence, the algorithm outputs "REJECT" with high probability in the soundness
  case.
\end{proof}

\samplecomplexitygap*
\begin{proof}
This is a consequence of sample complexity lower bound of $\Omega(d)$ from \Cref{lemma:lower_bound_truncation_near_accuracy}
and the robust mean estimation \cite[Proposition 1.20]{diakonikolas2023algorithmic} upper bound of $O(d)$.
\end{proof}

\section{Proof of \Cref{theorem:known_truncation_mean_testing}}

\meanvariancezknowntrunc*
\begin{proof}
By linearity of expectation and independence between $X_i$s and $Y_i$s,
\[ \mathbb{E} [Z_1] =\mathbb{E} \left[ \left( \frac{1}{n}  \sum_{i = 1}^n X_i
   -\tmmathbf{\mu}_S' \right)^T \right] \mathbb{E} \left[ \left( \frac{1}{n} 
   \sum_{i = 1}^n Y_i -\tmmathbf{\mu}_S' \right) \right] =
   (\tmmathbf{\mu}-\tmmathbf{\mu}_S')^T  (\tmmathbf{\mu}-\tmmathbf{\mu}_S') .
\]
Think of $\tilde{X}_i = X_i -\tmmathbf{\mu}_S'$ as a random variable
($\tilde{Y}_i = Y_i -\tmmathbf{\mu}_S'$ as the other random variable), denote
$\tilde{\Sigma}$ the covariance of a single $\tilde{X}_i$, and
$\tilde{\tmmathbf{\mu}}$ the mean. We have that as in the proof of \Cref{lm:z},
\[ \tmop{Var} \left[ \left( \frac{1}{n}  \sum_{i = 1}^n \tilde{X}_i \right)^T 
   \left( \frac{1}{n}  \sum_{i = 1}^n \tilde{Y}_i \right) \right] \leqslant
   \frac{\| \tilde{\Sigma} \|_F^2}{n^2} + \frac{2}{n}  \| \tilde{\Sigma} \|_F 
   \| \tilde{\tmmathbf{\mu}} \|_2^2 . \]
We know that shifting the location of a random variable does not affect the
covariance matrix, and thus $\tilde{\Sigma} = \Sigma_S$, and
$\tilde{\tmmathbf{\mu}} =\tmmathbf{\mu}_S -\tmmathbf{\mu}_S'$, which means
\begin{eqnarray*}
  \tmop{Var} [Z_1] & \leqslant & \frac{\| \Sigma_S \|_F^2}{n^2} + \frac{2}{n} 
  \| \Sigma_S \|_F  \| \tmmathbf{\mu}_S-\tmmathbf{\mu}_S' \|_2^2\\
  & \leqslant & \frac{(\| \Sigma_S -\tmmathbf{I}_d \|_F + \| \tmmathbf{I}_d
  \|_F)^2}{n^2} + \frac{2}{n}  (\| \Sigma_S -\tmmathbf{I}_d \|_F + \|
  \tmmathbf{I}_d \|_F)  \| \tmmathbf{\mu}_S-\tmmathbf{\mu}_S' \|_2^2\\
  & \lesssim & \frac{d}{n^2} + \frac{\sqrt{d}}{n}  \|
  \tmmathbf{\mu}_S-\tmmathbf{\mu}_S' \|_2^2,
\end{eqnarray*}
where the last step follows from \Cref{lm:truncatedvsnontruncated}. Letting $n = O \left(
\frac{\sqrt{d}}{\alpha^2} \right)$, we conclude our proof.
\end{proof}

\gapmeanknwontrunc*
\begin{proof}
Consider the negative log-likelihood function, $\bar{\ell}(\tmmathbf{0})$, with the mean set to $\tmmathbf{0}$ as the input parameter. This function is defined for a population drawn from a truncated normal distribution with an unknown mean ${\tmmathbf{\mu}}$. From \eqref{eq:gradient_of_negative_loglikelihood}, we can express the gradient of the negative log-likelihood with respect to the mean evaluated at $\tmmathbf{0}$, as follows:
    \[ \nabla \bar{\ell} (\tmmathbf{0}) = -\mathbb{E}_{\tmmathbf{x} \sim \mathcal{N}
   (\tmmathbf{\mu}, \tmmathbf{I}_d, S)} [\tmmathbf{x}]
   +\mathbb{E}_{\tmmathbf{z} \sim \mathcal{N} (\tmmathbf{0}, \tmmathbf{I}_d,
   S)} [\tmmathbf{z}] =\tmmathbf{\mu}_S -\tmmathbf{\mu}_S' . \]
   Likewise, when evaluating the gradient at $\tmmathbf{\mu}$, we have
   \[ \nabla \bar{\ell} (\tmmathbf{\mu}) = -\mathbb{E}_{\tmmathbf{x} \sim \mathcal{N}
   (\tmmathbf{\mu}, \tmmathbf{I}_d, S)} [\tmmathbf{x}]
   +\mathbb{E}_{\tmmathbf{z} \sim \mathcal{N} (\tmmathbf{\mu}, \tmmathbf{I}_d,
   S)} [\tmmathbf{z}] = \tmmathbf{0} . \]
   So, $\nabla\bar{\ell}(\tmmathbf{0})$ represents the difference between the truncated mean of the underlying distribution and that of the distribution with mean $\tmmathbf{0}$. From \Cref{lemma:strong_convexity_negative_loglikelihood}, let $\lambda_0 = \frac{1}{2^{13}}  \left( \frac{\beta}{C} \right)^4
   \min \left\{ \frac{1}{4}, \frac{1}{16 \| \tmmathbf{\mu} \|_2^2 + 1} \right\}$, we know that $\bar{\ell}(\cdot)$ is $\lambda_0$-strongly convex, and $\lambda_0$ is a constant if $\beta$ is a constant and $\| \tmmathbf{\mu} \|^2_2 \leqslant \frac{1}{16}$.
   Therefore, by leveraging the properties of strong convexity and applying the Cauchy–Schwarz inequality, we obtain the following result:
   \[ \sqrt{\|\tmmathbf{\mu}-\tmmathbf{0}\|_2^2 \cdot \| \nabla \bar{l}
   (\tmmathbf{\mu}) - \nabla \bar{l} (\tmmathbf{0})\|_2^2} \geqslant \langle
   \nabla \bar{l} (\tmmathbf{\mu}) - \nabla \bar{l} (\tmmathbf{0}),
   \tmmathbf{\mu}-\tmmathbf{0} \rangle \geqslant \frac{\lambda_0}{2}
   \|\tmmathbf{\mu}\|_2^2 \]
By simplifying the expression and substituting $\tmmathbf{\mu}$ with any $\| \tmmathbf{\mu''} \|_2^2 \geqslant \alpha^2$, we can show that:
\[
    \| \tmmathbf{\mu}_S'' -\tmmathbf{\mu}_S' \|_2^2 \geqslant \Omega (\alpha^2).
\]
\end{proof}

\knownTruncationAlgoTheorem*
\begin{proof}
Suppose $\| \tmmathbf{\mu} \|_2^2 = \alpha^2 \leqslant \frac{1}{16}$, then we have
\[ \min
   \left\{ \frac{1}{4}, \frac{1}{16 \| \mu \|_2^2 + 1} \right\} = \frac{1}{4}
   . \]
And so $\lambda_0$ is a constant.

\paragraph{Completeness:}When $P$ is a truncated Gaussian $\mathcal{N}
(\tmmathbf{0}, \tmmathbf{I}_d, S)$. By 
\cref{lemma:expectation_variance_with_known_truncation}, we have that
when $n \geqslant \Omega \left( \frac{\sqrt{d}}{\lambda_0^2 \alpha^2} \right)
= \Omega \left( \frac{\sqrt{d}}{\alpha^2} \right)$,
\[ \mathbb{E} [Z_1] =\tmmathbf{0} \infixand \tmop{Var} [Z_1] \leqslant O
   (\alpha^4) . \]
By Chebyshev's inequality, for $n$ being a large enough multiple of
$\frac{\sqrt{d}}{\alpha^2}$, we have that,
\[ \Pr \left[ Z_1 \geqslant \frac{3}{2} c_2 \cdot \alpha^2 \right] \leqslant O
   \left( \frac{\tmop{Var} [Z_1]}{c_2 \cdot \alpha^4} \right) \leqslant
   \frac{1}{10} . \]
\paragraph{Soundness:}When $P$ is a truncated Gaussian $\mathcal{N}
(\tmmathbf{\mu}, \tmmathbf{I}_d, S)$ and $\| \tmmathbf{\mu} \|_2^2 \geqslant
\alpha^2$. By~\cref{lemma:gap_truncated_mean}, we have that the gap under truncation is:
\[ \| \tmmathbf{\mu}_S -\tmmathbf{\mu}_S' \|_2^2 \geqslant \Omega (\alpha^2) =
   c_2 \cdot \alpha^2 . \]

By \cref{lemma:expectation_variance_with_known_truncation}, we have that
when $n \geqslant \Omega \left( \frac{\sqrt{d}}{\lambda_0^2 \alpha^2} \right)
= \Omega \left( \frac{\sqrt{d}}{\alpha^2} \right)$,
\[ \mathbb{E} [Z_1] = \| \tmmathbf{\mu}-\tmmathbf{\mu}_S \|_2^2 \infixand
   \tmop{Var} [Z_1] \leqslant O (\alpha^4 + \alpha^2  \| \tmmathbf{\mu}_S
   -\tmmathbf{\mu}_S' \|_2^2) . \]
By Chebyshev's inequality, we have that,
\[ \Pr \left[ Z_1 \leqslant \frac{3}{2} c_1 \cdot \alpha^2 \right] \leqslant
   \Pr \left[ Z_1 \leqslant \| \tmmathbf{\mu}-\tmmathbf{\mu}_S \|_2^2 +
   \frac{1}{2} c_2 \nobracket \cdot \alpha^2 \nobracket \right] \leqslant
   \frac{\tmop{Var} [Z_1]}{\left( \frac{1}{2} c_2 \nobracket \cdot \alpha^2
   \nobracket \right)^2} \leqslant O \left( \frac{\alpha^4 + \alpha^2  \|
   \tmmathbf{\mu}_S -\tmmathbf{\mu}_S' \|_2^2}{c_2^2 \nobracket \cdot \alpha^4
   \nobracket} \right) . \]
Let $n$ be a large enough multiple of $\frac{\sqrt{d}}{\alpha^2}$, then
\[ \Pr \left[ Z_1 \leqslant \frac{3}{2} c_1 \cdot \alpha^2 \right] \leqslant O
   \left( \frac{\alpha^4 + \alpha^2  \| \tmmathbf{\mu}_S -\tmmathbf{\mu}_S'
   \|_2^2}{c_2^2 \cdot \alpha^4} \right) = O \left( \frac{\alpha^4 + c_2 \cdot
   \alpha^4}{c_2^2 \cdot \alpha^4} \right) \leqslant \frac{1}{10} . \]   
\end{proof}

\section{Proof of \Cref{lemma:lower_bound_truncation_near_accuracy}}
\label{appendix:sec:chi_square_upper_bound}
\lbunknowntruncation*
\begin{proof}
We begin by constructing a one-dimensional truncated normal distribution $A = \mathcal{N} (\alpha, 1, S)$, where the truncated mass is $\eps$. This means 
$\Pr_{x \sim \mathcal{N} (\alpha, 1)} [x \in S] = 1 - \eps$. We can determine the ($1 - \eps$)-quantile as:
 \[ 
     b = \alpha + \sqrt{2} \tmop{erf}^{- 1} (1 - 2 \eps) . 
 \]
which defines the truncation set as 
$S \coloneqq (- \infty, b]$.

Let $\alpha (\eps) \coloneqq \alpha = \Theta \left( \eps \sqrt{\log \frac{1}{\eps}} \right)$. For any $\eps$, we can find a constant $c_2 = \Theta(1)$ such that $\mathbb{E} [A] = 0$:
\[
    \mathbb{E}_{X \sim A} [X] 
    =
    \alpha - \frac{\exp \left( - \frac{1}{2} 
    \left( b - \alpha \right)^2 \right)}{\sqrt{2 \pi}  (1 -
    \varepsilon)} 
    = 0, 
\]
which is equivalent to:
\[ 
    \alpha = \frac{\exp (- (\tmop{erf}^{- 1} (1 - 2 \eps))^2)}{\sqrt{2 \pi}  (1 - \eps)} 
    =
    \Theta \left( \eps \sqrt{\log \frac{1}{\eps}} \right). 
\]
Next, we compute an upper bound on the chi-squared divergence between the truncated distribution $A$ and the standard normal distribution $\mathcal{N}(0, 1)$. We find that 
\begin{eqnarray*}
  \chi^2 (A, \mathcal{N} (0, 1)) & = & \left( \int_{- \infty}^b \frac{\left( \exp \left( - \frac{(x -
  \alpha)^2}{2} \right) / \sqrt{2 \pi}  (1 - \varepsilon) \right)^2}{\exp
  \left( - \frac{x^2}{2} \right) / \sqrt{2 \pi}} d x \right) - 1\\
  & = & \frac{1}{\sqrt{2 \pi}  (1 - \varepsilon)^2} \left( \int_{- \infty}^b
  \exp (- (x - \alpha)^2 + x^2 / 2) d x \right) - 1\\
  & = & \frac{1}{\sqrt{2 \pi}  (1 - \varepsilon)^2} \left( \int_{- \infty}^b
  \exp \left( - \frac{x^2}{2} + 2 x \alpha - \alpha^2 \right) d x \right) -
  1\\
  & = & \frac{1}{\sqrt{2 \pi}  (1 - \varepsilon)^2} \left( \int_{- \infty}^b
  \exp \left( - \left( \frac{x}{\sqrt{2}} \right)^2 + 2 x \alpha - \left(
  \sqrt{2} \alpha \right)^2 + \alpha^2 \right) d x \right) - 1\\
  & = & \frac{\exp (\alpha^2)}{\sqrt{2 \pi}  (1 - \varepsilon)^2} \left(
  \int_{- \infty}^b \exp \left( - \left( \frac{x}{\sqrt{2}} - \sqrt{2} \alpha
  \right)^2 \right) d x \right) - 1\\
  & = & \frac{\exp (\alpha^2)}{\sqrt{2 \pi}  (1 - \varepsilon)^2} \left(
  \int_{- \infty}^b \exp \left( - \frac{(x - 2 \alpha)^2}{2} \right) d x
  \right) - 1\\
  & = & \frac{\mathcal{N} (2 \alpha, 1 , S)}{(1 - \eps)^2}
  \cdot \exp (\alpha^2) - 1\\
  & \leqslant & \frac{\exp (\alpha^2)}{(1 - \varepsilon)^2} - 1\\
  & \leqslant & (1 + O (\varepsilon)) \cdot (1 + O (\alpha^2)) - 1\\
  & \leqslant & O (\varepsilon) + O (\alpha^2) .
\end{eqnarray*}
We now apply \cref{prop:lbtesting} \cite[Proposition 7.1]{DBLP:journals/corr/DiakonikolasKS16c}, and obtain a lower bound of
\[ 
    \Omega \left( \frac{d}{\eps + \alpha^2} \right) 
    =
    \Omega \left(\frac{d}{\eps} \right). 
\]
concluding the proof.
\end{proof}

\vfill

\end{document}